\documentclass[12pt]{article}
\usepackage{amsmath,amsthm,amssymb,latexsym,color}

\theoremstyle{plain}
\newtheorem{theor}{Theorem}
\theoremstyle{remark}
\newtheorem{rem}{Remark}
\theoremstyle{plain}
\newtheorem{prop}[theor]{Proposition}

\newtheorem{lemma}[theor]{Lemma}

\def\R{{\mathbb R}}
\def\P{{\mathbb P}}
\def\N{{\mathbb N}}

\def\Net{{\mathcal N}}

\def\Exp{{\mathbb E}}
\def\M{{\mathcal M}}

\def\Proj{{\rm Proj}}

\begin{document}

\title{The limit of the smallest singular value\\of random matrices with i.i.d.\ entries}
\author{Konstantin Tikhomirov\\
\small Department of Mathematical and Statistical Sciences, University of Alberta\\
\small Edmonton, Alberta, T6G 2G1, Canada}

\maketitle

\begin{abstract}
Let $\{a_{ij}\}$ $(1\le i,j<\infty)$ be i.i.d.\ real valued random variables with zero mean and unit variance
and let an integer sequence $(N_m)_{m=1}^\infty$ satisfy $m/N_m\longrightarrow z$ for some $z\in(0,1)$.
For each $m\in\N$ denote by $A_m$ the $N_m\times m$ random matrix $(a_{ij})$ $(1\le i\le N_m,1\le j\le m)$
and let $s_{m}(A_m)$ be its smallest singular value.
We prove that the sequence $\bigl({N_m}^{-1/2} s_{m}(A_m)\bigr)_{m=1}^\infty$ converges to $1-\sqrt{z}$ almost surely.
Our result does not require boundedness of any moments of $a_{ij}$'s higher than the $2$-nd and
resolves a long standing question regarding the weakest moment assumptions on the distribution of the entries
sufficient for the convergence to hold.
\end{abstract}

\section{Introduction}

For $N\ge m$ and an $N\times m$ real-valued matrix $B$, its {\it singular values} $s_1(B)$, $s_2(B),\dots$, $s_m(B)$
are the eigenvalues of the matrix $\sqrt{B^T B}$ arranged in non-increasing order,
where multiplicities are counted.
In particular, {\it the largest} and {\it the smallest} singular values are given by
$$s_1(B)=\sup\limits_{y\in S^{m-1}}\|By\|=\|B\|;\;\;\;s_m(B)=\inf\limits_{y\in S^{m-1}}\|By\|.$$
In this paper, we establish convergence of the smallest singular values of a sequence random matrices with i.i.d.\ entries
under minimal moment assumptions.

The extreme singular values of random matrices
attract considerable attention of researchers both in {\it limiting} and {\it non-limiting} settings.
We refer the reader to surveys and monographs \cite{BS}, \cite{PS}, \cite{RV_CONGRESS}, \cite{V}
for extensive information on the spectral theory of random matrices.
Here, we shall focus on the following specific question:
for matrices with i.i.d.\ entries, what are the weakest possible assumptions
on the entries which are sufficient for the smallest singular value to ``concentrate''?

We note that a corresponding problem for the {\it largest} singular value (i.e.\ the operator norm) was
essentially resolved in the i.i.d.\ case, where finiteness of the fourth moment of
the entries turns out to be crucial both in limiting and non-limiting settings. We refer the reader to
\cite{YBK} and \cite{BSY} for results on a.s.\ convergence of the largest singular value, and
\cite{L} for the non-limiting case (see also \cite{S}, \cite{LS} for some negative results
on concentration of the operator norm).

For the {\it smallest}
singular value, its concentration properties are relatively well understood in the i.i.d.\ case
provided that the fourth moment of the matrix entries is bounded.
A classical theorem of Bai and Yin \cite{BY} (see also \cite[Theorem~5.11]{BS}) states the following:
given an array $\{a_{ij}\}$ $(1\le i,j<\infty)$ of i.i.d.\ random variables such that $\Exp a_{ij}=0$, $\Exp {a_{ij}}^2=1$
and $\Exp {a_{ij}}^4<\infty$, and an integer sequence $(N_m)_{m=1}^\infty$ with $m/N_m\longrightarrow z$ for some $z\in(0,1)$,
the $N_m\times m$ matrices $A_{m}=(a_{ij})$ $(1\le i\le N_m,1\le j\le m)$ satisfy
$${N_m}^{-1/2}s_{m}(A_{m})\longrightarrow 1-\sqrt{z}\;\;\mbox{almost surely}.$$
Further, it is proved in \cite{RV_UPPER BOUND}, \cite{RV_SQUARE} that
for square $m\times m$ matrices with i.i.d.\ centered entries with unit variance and
a bounded fourth moment, one has $s_m(A)\approx m^{-1/2}$ with a large probability.

A natural question in connection with the mentioned results is {\it whether the assumption on the
fourth moment is necessary for the least singular value to ``concentrate''};
in particular, whether
any assumptions on moments of $a_{ij}$'s higher than the $2$-nd are required for the a.s.\ convergence
in the Bai--Yin theorem. This question is discussed in \cite{BS} on p.~6.
Solving the problem was a motivation for our work.

A considerable progress has been made recently in the direction of weakening the moment assumptions on matrix entries.
For square matrices, given a sufficiently large $m$ and an $m\times m$ matrix with i.i.d.\ entries with zero mean
and unit variance, its smallest
singular value is bounded from below by a constant (negative) power of $m$ with probability close to one
\cite[Theorem~2.1]{TV} (see also \cite[Theorem~4.1]{GT} for sparse matrices).

For {\it tall} rectangular matrices, Srivastava and Vershynin proved in \cite{SV} that for any $\varepsilon,\eta>0$ and
an $N\times m$ random matrix $A$ with independent isotropic
rows $X_i$ such that $\sup\limits_{y\in S^{m-1}}\Exp|\langle X_i,y\rangle|^{2+\eta}\le C$,
the singular value $s_m(A)$ satisfies $\Exp s_m(A)^2\ge (1-\varepsilon)N$ provided that
the aspect ratio $N/m$ is bounded from below by a certain function of $\varepsilon$ and $\eta$.
This result of \cite{SV} was strengthened by Koltchinskii and Mendelson \cite{KM} who proved that, under similar assumptions
on the matrix, $s_m(A)\ge (1-\varepsilon)\sqrt{N}$ with a very large probability. Moreover,
another theorem of \cite{KM} states that, for a sufficiently tall $N\times m$ random matrix $A$
with i.i.d.\ isotropic rows satisfying certain ``spreading'' condition,
$s_m(A)\gtrsim\sqrt{N}$ with probability very close to one.
Some further strengthening of the results of \cite{KM} is obtained in \cite{Yaskov}.

A situation when no upper bounds for moments of the matrix entries are given, was considered in \cite{T}.
It was proved that for any $\delta>1$, $N\ge\delta m$ and
for an $N\times m$ random matrix $A$ with i.i.d.\ entries satisfying
$\inf\limits_{\lambda\in\R}\P\bigl\{|a_{11}-\lambda|\ge\alpha\bigr\}\ge \beta$ for some $\alpha,\beta>0$,
one has $\P\bigl\{s_m(A)\ge \alpha u\sqrt{N}\bigr\}\ge 1-2\exp(-vN)$,
where $u,v>0$ depend only on $\beta$ and $\delta$.

The result of \cite{T} can be used to show that in the limiting setup
of the Bai--Yin theorem but without the assumptions on moments higher than the $2$-nd,
the sequence $\bigl({N_m}^{-1/2}s_{m}(A_{m})\bigr)_{m=1}^\infty$ satisfies
$$\liminf\limits_{m\to\infty}\bigl({N_m}^{-1/2}s_{m}(A_{m})\bigr)\ge r>0\;\;\mbox{almost surely},$$
where $r$ is a certain function of $z=\lim m/N_m$ and the distribution of $a_{ij}$'s.
The same conclusion can be derived from \cite[Theorem~1.4]{KM}, if
we additionally assume that the limiting aspect ratio $z$ is bounded from above
by a sufficiently small positive quantity (i.e.\ the matrices are tall).
However, both \cite[Theorem~1]{T} and \cite[Theorem~1.4]{KM} do not give
the precise asymptotics.

This problem is resolved in our paper. The main result is the following
\begin{theor}\label{universal theor}
Let $\{a_{ij}\}$ $(1\le i,j<\infty)$ be a set of i.i.d.\ real valued random variables with zero mean and unit variance.
Further, let $(N_m)_{m=1}^\infty$ be an integer sequence satisfying $m/N_m\longrightarrow z$ for some $z\in(0,1)$.
For every $m\in\N$ we denote by $A_{m}$ the random $N_m\times m$ matrix with entries $a_{ij}$
$(1\le i\le N_m,1\le j\le m)$. Then with probability one the sequence
$$\bigl({N_m}^{-1/2}s_{m}(A_{m})\bigr)_{m=1}^\infty$$
converges to $1-\sqrt{z}$.
\end{theor}

Theorem~\ref{universal theor} in a strong form establishes the {\it asymmetry} of the limiting behaviour of the extreme singular values:
whereas the fourth moment is necessary for the operator norm, the second moment is sufficient for the convergence of the smallest singular value.

Let us briefly describe our approach to proving Theorem~\ref{universal theor}.
We shall ``approximate'' the matrices $A_m$ by
matrices with truncated and centered entries. Namely, for $M>0$ and all $m\ge 1$ let
$\tilde A_m$ be the $N_m\times m$ matrix with the entries
$$\tilde a_{ij}=a_{ij}\chi_{\{|a_{ij}|\le M\}}-\Exp(a_{ij}\chi_{\{|a_{ij}|\le M\}}),\;\;1\le i\le N_m,\;\;1\le j\le m,$$
where $\chi_{\mathcal E}$ is the indicator of an event $\mathcal E$.
If the truncation level $M$ is large enough then it turns out that for {\it all sufficiently large $m$}
we have $s_m(\tilde A_m)\approx s_m(A_m)$ with probability close to one.
In fact, we need only one-sided estimate for our proof. To be more precise, we will show that with a large probability the quantity
$$\limsup\limits_{m\to\infty}{N_m}^{-1/2}\bigl(s_m(\tilde A)-s_m(A)\bigr)$$
is bounded from above by a
positive number which depends only on $M$ and can be made arbitrarily small by increasing the truncation level
(in a more technical form, this is stated in Theorem~\ref{trunc ssv theor} of the note).
Then, applying the Bai--Yin theorem \cite{BY} to the truncated matrices $\tilde A_m$,
we get
$$\liminf\limits_{m\to\infty}{N_m}^{-1/2}s_m(A)\gtrsim \liminf\limits_{m\to\infty} {N_m}^{-1/2}s_m(\tilde A)\gtrsim 1-\sqrt{z}\;\;\mbox{almost surely},$$
which implies the result.
Thus, the argument of the paper \cite{BY} remains the crucial element of the proof, although we apply it
only to the truncated variables, for which all positive moments are bounded.
Let us emphasize that, whereas a truncation procedure for matrices also appears as a technical step in \cite{BY},
in our approach the truncation level $M$ is {\it not a function of $m$}.

Note that the equivalence $s_m(A_m)\approx s_m(\tilde A_m)$ would follow immediately
if the difference $A_m-\tilde A_m$ had the operator norm very small compared to $\sqrt{N_m}$
with a large probability. However, the moment assumptions that we impose on $a_{ij}$'s
are too weak to expect a good upper bound for $\|A_m-\tilde A_m\|$.
To overcome this problem, we shall consider a special {\it non-convex} function
of the matrix $A_m-\tilde A_m$ which has much better concentration properties
than the norm and which shall act as a ``replacement'' for the norm in our calculations.
This quantity and its concentration properties are discussed in Section~\ref{weak lsv section}
and are the main novel igredient of the paper.

\section{Preliminaries}

In this section, we introduce notation and present some classical or elementary facts, which we
include for an easier referencing.

We denote by $(\Omega,\Sigma,\P)$ a probability space,
and adopt the usual notations and definitions from the Probability Theory
such as i.i.d.\ random variables, the expectation, etc.
Let $\{e_i\}_{i=1}^N$ be the standard unit vector basis in $\R^N$,
$\|\cdot\|$ and $\langle\cdot,\cdot\rangle$ be the canonical Euclidean norm
and corresponding inner product, and $\|\cdot\|_\infty$ be the maximum ($\ell_\infty$-) norm.
The unit Euclidean ball in $\R^n$ shall be denoted by $B_2^n$ and
the cube $[-1,1]^n$ ~--- by $B_\infty^n$.
For a finite set $I$, $|I|$ is its cardinality.
Universal constants are denoted by $C,c_1$, etc.
A numerical subscript in the name of a constant determines the statement where the constant
is defined. Similarly, a function defined within a statement and intended to be
used further in the paper, has the statement number as a subscript.

Let $T$ be a subset of $\R^n$ and $\|\cdot\|_B$ be a norm on $\R^n$ with the unit ball $B$.
A subset $\Net\subset T$ is called {\it an $\varepsilon$-net in $T$ with respect to $\|\cdot\|_B$}
if for any $y\in T$ there is $y'\in\Net$ satisfying $\|y-y'\|_B\le\varepsilon$.
We shall omit the reference to $\|\cdot\|_B$ when $B=B_2^n$.

\begin{lemma}\label{usual net}
For any $n\in\N$ and $\varepsilon\in(0,1]$ there exists an $\varepsilon$-net in $B_2^n$
of cardinality at most $\bigl(\frac{3}{\varepsilon}\bigr)^n$.
\end{lemma}

\begin{lemma}\label{cubic net}
For any $n\in\N$ and any $T\subset S^{n-1}$ there is an $n^{-1/2}$-net
in $T$ with respect to $\|\cdot\|_\infty$ of cardinality at most $\exp(C_{\ref{cubic net}} n)$.
Here, $C_{\ref{cubic net}}>0$ is a universal constant.
\end{lemma}

\begin{rem}
Both lemmas above follow from a well known estimate for covering numbers
for pairs of convex sets in $\R^n$ (see, for example, {\cite[Lemma~4.16]{Pisier}}).
For Lemma~\ref{cubic net},
the estimate for the pair $(B_2^n,B_\infty^n)$ yields an existence of a
$(4n)^{-1/2}$-net $\bar\Net$ in $B_2^n$ with respect to $\|\cdot\|_\infty$
of cardinality at most $\exp(C_{\ref{cubic net}} n)$ for an absolute constant
$C_{\ref{cubic net}}>0$. Then $\Net\subset T$ can be constructed by picking a point
from every non-empty intersection of the form $(y'+(4n)^{-1/2}B_\infty^n)\cap T$, $y'\in\bar\Net$.
\end{rem}

The next statement, which is sometimes called the Bernstein (or Hoeffding's) inequality, can be derived
from classical Khintchine's inequality for the sum of weighted independent signs
by a symmetrization procedure:
\begin{lemma}[see, for ex., {\cite[Proposition~5.10]{V}}]\label{Khintchine mod}
Let $n\in\N$, $M>0$, $y=(y_1,y_2,\dots,y_n)$ with $\|y\|=1$, and let $a_1,a_2\dots,a_n$ be
independent mean zero random variables with $|a_j|\le M$ a.s.\ ($j=1,2,\dots,n$). Then
$$\P\Bigl\{\Bigl|\sum\limits_{j=1}^n a_j y_j\Bigr|\ge\tau\Bigr\}\le 2\exp(-c_{\ref{Khintchine mod}}\tau^2/M^2),\;\;\tau>0,$$
where $c_{\ref{Khintchine mod}}>0$ is a universal constant.
\end{lemma}

The lemma below is a law of large numbers, where instead of the arithmetic mean of a collection
of random variables we consider more general weighted sums.
As in the case of the classical weak LLN, the statement can be proved by applying
Levy's continuity theorem for characteristic functions.
\begin{lemma}\label{law of ln}
Let $a_1,a_2,\dots$ be i.i.d.\ random variables with zero mean. Then for any $\varepsilon>0$ there
is $\delta>0$ depending only on $\varepsilon$ and the distribution of $a_j$'s with the following property:
whenever $(t_j)_{j=1}^\infty$ is a sequence of non-negative real numbers such that $\sum_{j=1}^\infty t_j=1$
and $\max t_j\le\delta$, we have
$$\P\Bigl\{\Bigl|\sum\limits_{j=1}^\infty a_j t_j\Bigr|>\varepsilon\Bigr\}<\varepsilon.$$
\end{lemma}

Given an $m\times m$ random symmetric matrix $T$ with eigenvalues $\lambda_1,\lambda_2,\dots,\lambda_m$,
{\it the empirical spectral distribution} of $T$ is the function on $\R$ given by
$$F^T(t)=\frac{1}{m}\bigl|\bigl\{j\le m:\,\lambda_j\le t\bigr\}\bigr|,\;\;t\in\R.$$

\begin{theor}[{Mar\v cenko--Pastur law; see \cite{MP}, \cite{Y}, \cite[Theorem~3.6]{BS}}]\label{Mar Pas}
Let $\{a_{ij}\}$ $(1\le i,j\le\infty)$ be a set of i.i.d.\ random variables with zero mean and unit variance and let
$(N_m)_{m=1}^\infty$ be an integer sequence satisfying $m/N_m\longrightarrow z$ for some $z\in(0,1)$.
For every $m\in\N$ denote by $A_{m}$ the random $N_m\times m$ matrix with entries $a_{ij}$
$(1\le i\le N_m,1\le j\le m)$ and by $T_m$ the matrix $\frac{1}{N_m}A_{m}^T A_{m}$.
Then with probability one the sequence of empirical spectral distributions $\{F^{T_m}\}$
converges pointwise to a non-random distribution
given by
$$F_{MP}(t)=
\begin{cases}
0,&\mbox{if }t\le r,\\
\frac{1}{2\pi z}\int\limits_{r}^t \frac{\sqrt{(R-\tau)(\tau-r)}}{\tau}\,d\tau,&\mbox{if }r\le t\le R,\\
1,&\mbox{if }t\ge R.
\end{cases}$$
where $r=(1-\sqrt{z})^2$ and $R=(1+\sqrt{z})^2$.
\end{theor}

\begin{rem}\label{Mar Pas rem}
Note that the above theorem does not require any assumptions on moments higher than the $2$nd,
and so can be applied in our setting.
For our proof, we will actually need a much weaker result
than Theorem~\ref{Mar Pas}, namely, that $\limsup\limits_{m\to\infty}\frac{s_m(A_m)}{\sqrt{N_m}}\le 1-\sqrt{z}$
almost surely.
The latter can be immediately verified with help of Theorem~\ref{Mar Pas}: for every fixed $t> (1-\sqrt{z})^2$,
we have $\lim\limits_{m\to\infty}F^{T_m}(t)=F_{MP}(t)>0$ with probability one, hence
the smallest non-zero eigenvalues $\lambda_{\min}(T_m)$ of matrices $T_m$ satisfy
$\limsup\limits_{m\to\infty}\lambda_{\min}(T_m)\le t$ a.s. This implies
$\limsup\limits_{m\to\infty}\frac{s_m(A_m)}{\sqrt{N_m}}\le \sqrt{t}$ a.s., which gives
the required estimate by letting $t\to (1-\sqrt{z})^2$.
\end{rem}

\section{Norms of coordinate projections of random vectors}\label{weak lsv section}

For any $N\in\N$ and a subset $I\subset\{1,2,\dots,N\}$, let us denote by $\Proj_I:\R^N\to\R^N$
the coordinate projection onto the subspace spanned by $\{e_i\}_{i\in I}$.
Throughout the rest of the paper, we will often use expressions of the form $\min\limits_{|I|\ge r}\|\Proj_I x\|$,
where $x$ is some vector in $\R^N$ and $r$ is a positive real number. This notation should be interpreted as
the minimum of $\|\Proj_I x\|$ over {\it all} subsets $I\subset\{1,2,\dots,N\}$ of cardinality at least $r$.

The goal of this section is to show that, given a sufficiently large random $N\times n$ matrix $A$ with i.i.d.\
entries with zero mean and unit variance, the quantity
\begin{equation}\label{the quantity}
\sup\limits_{y\in S^{n-1}}\min\limits_{|I|\ge N-\varepsilon N}\|\Proj_{I}Ay\|
\end{equation}
is of order $\sqrt{N}$ with a very large probability (the probability shall depend on $\varepsilon>0$).
It shall act as a ``replacement'' of the matrix norm $\|A\|$ which 
in our setting may be greater than $\sqrt{N}$ by the order of magnitude with probability close to one.
We remark here that a quantity
$$\max\limits_{|I|=m}\|\Proj_I D\|=\sup\limits_{y\in S^{n-1}}\max\limits_{|I|=m}\|\Proj_I Dy\|,$$
where $m\le N$ and $D$ is an $N\times n$ random matrix with i.i.d.\ isotropic log-concave rows,
played a crucial role in the paper~\cite{ALPT} by Adamczak, Litvak, Pajor and Tomczak-Jaegermann,
dealing with the problem of approximating covariance matrix of a log-concave random vector
by the sample covariance matrix. In our case, however, the latter
quantity is inapplicable as it may not concentrate near $\sqrt{N}$ (even for small $m$).

First, we prove the required estimate for \eqref{the quantity} under the additional assumption
that the entries of $A$ are symmetrically distributed (Lemma~\ref{weak lsv sym}).
Then we generalize the result to non-symmetric distributions in Proposition~\ref{weak lsv nonsym}.
Lemmas~\ref{single vector est}--\ref{cube bound} given below build the framework of the proof.

\begin{lemma}\label{single vector est}
For each $\varepsilon\in(0,1]$ there is $N_{\ref{single vector est}}=N_{\ref{single vector est}}(\varepsilon)>0$
depending only on $\varepsilon$
with the following property:
let $N\ge N_{\ref{single vector est}}$
and let $X=(X_1,X_2,\dots,X_N)$ be a random vector of independent variables, each
$X_i$ having zero mean and unit variance.
Then
$$\min\limits_{|I|\ge N-\varepsilon N}\|\Proj_{I}X\|\le C_{\ref{single vector est}}\sqrt{N}$$
with probability at least $1-\exp(-c_{\ref{single vector est}}\varepsilon N)$, where $C_{\ref{single vector est}},c_{\ref{single vector est}}>0$
are universal constants.
\end{lemma}
\begin{proof}
Fix any $\varepsilon\in(0,1]$ and define $N_{\ref{single vector est}}$ as the smallest positive integer such that
$$\Bigl(\frac{e}{4}\Bigr)^{\varepsilon N}+\exp(-\varepsilon e N/4)\le\exp(-\varepsilon N/3)$$
for all $N\ge N_{\ref{single vector est}}$. Choose any $N\ge N_{\ref{single vector est}}$ and let 
$X$ be as stated above. Set $M=\frac{4}{\varepsilon}$.
In view of Markov's inequality,
$$\P\bigl\{|\{i\le N:\,|X_i|\ge \sqrt{M}\}|\ge 4N/M\bigr\}\le {N\choose \lceil 4N/M\rceil}\Bigl(\frac{1}{M}\Bigr)^{\lceil 4N/M\rceil}
\le \Bigl(\frac{e}{4}\Bigr)^{\lceil 4N/M\rceil}.$$
Let $\tilde X=(\tilde X_1,\tilde X_2,\dots,\tilde X_N)$ be a vector of truncations of $X_i$'s, with
$$\tilde X_i(\omega)=\begin{cases}X_i(\omega),&\mbox{if }|X_i(\omega)|\le \sqrt{M};\\0,&\mbox{otherwise}.\end{cases}$$
Then, from the above estimate,
$$\P\bigl\{\min\limits_{|I|\ge N-\varepsilon N}\|\Proj_{I}X\|>\|\tilde X\|\bigr\}
\le \P\bigl\{|\{i\le N:\,|X_i|\ge \sqrt{M}\}|\ge \varepsilon N\bigr\}\le \Bigl(\frac{e}{4}\Bigr)^{\varepsilon N}.$$
Now, let us estimate the Euclidean norm of $\tilde X$ using the Laplace transform. Set $\lambda=\frac{1}{M}$. We have
\begin{align*}
\Exp\exp(\lambda\|\tilde X\|^2)&=\prod\limits_{i=1}^N\Exp\exp(\lambda \tilde X_i^2)\\
&=\prod\limits_{i=1}^N\Bigl(1+\int_1^{\exp(\lambda M)}\P\bigl\{\exp(\lambda \tilde X_i^2)\ge\tau\bigr\}\,d\tau\Bigr)\\
&\le\prod\limits_{i=1}^N\Bigl(1+\int_1^{e}\P\Bigl\{\tilde X_i^2\ge\frac{\tau-1}{e\lambda}\Bigr\}\,d\tau\Bigr)\\
&\le\prod\limits_{i=1}^N\Bigl(1+e\lambda\Exp \tilde X_i^2\Bigr)\\
&\le\bigl(1+e\lambda\bigr)^N\\
&\le\exp(eN/M).
\end{align*}
Hence,
$$\P\bigl\{\|\tilde X\|\ge \sqrt{2eN}\bigr\}\le\exp(-eN/M).$$
Finally, using the definition of $N_{\ref{single vector est}}$, we get
\begin{align*}
\P\bigl\{\min\limits_{|I|\ge N-\varepsilon N}\|\Proj_{I}X\|>\sqrt{2eN}\bigr\}
&\le\P\bigl\{\min\limits_{|I|\ge N-\varepsilon N}\|\Proj_{I}X\|>\|\tilde X\|\bigr\}+\P\bigl\{\|\tilde X\|\ge \sqrt{2eN}\bigr\}\\
&\le\Bigl(\frac{e}{4}\Bigr)^{\varepsilon N}+\exp(-\varepsilon e N/4)\\
&\le\exp(-\varepsilon N/3).
\end{align*}
\end{proof}

\begin{lemma}\label{vector bound sym}
For every $K>0$ there is $L_{\ref{vector bound sym}}=L_{\ref{vector bound sym}}(K)>0$ depending only on $K$ with the following property:
Let $N,n\in\N$, $N\ge n$, and let $A=(a_{ij})$ be an $N\times n$ random matrix with i.i.d.\ symmetrically distributed
entries with unit variance. For each $y=(y_1,y_2,\dots,y_n)\in S^{n-1}$ let $I_y:\Omega\to 2^{\{1,2,\dots,N\}}$
be a random subset of $\{1,2,\dots,N\}$ defined as
$$I_y=\Bigl\{i\le N:\,\sum\limits_{j=1}^n a_{ij}^2 y_j^2\le 2\Bigr\}.$$
Then for every $y\in S^{n-1}$ we have
$$\P\bigl\{\|\Proj_{I_y}Ay\|\ge  L_{\ref{vector bound sym}}\sqrt{N}\bigr\}\le\exp(-KN).$$
\end{lemma}
\begin{proof}
Fix any $K>0$ and let $N,n$ and $A=(a_{ij})$ be as stated above.
Let $r_{ij}$ $(1\le i\le N,\;1\le j\le n)$ be Rademacher variables jointly independent with $A$, and
let $\bar A$ denote the random $N\times n$ matrix $(r_{ij}a_{ij})$. Then, since $a_{ij}$'s are symmetrically distributed,
for any fixed vector $y=(y_1,y_2,\dots,y_n)\in S^{n-1}$ the distribution of $\|\Proj_{I_y}Ay\|$
is the same as that of $\|\Proj_{I_y}\bar Ay\|$. Define a subset of (non-random) $N\times n$ matrices:
$$\M_y=\Bigl\{B=(b_{ij})\in\R^{N\times n}:\,\sum\limits_{j=1}^n b_{ij}^2y_j^2\le 2\mbox{ for all }i=1,2,\dots,N\Bigr\}$$
and for every $B=(b_{ij})\in\M_y$ denote by $\bar B$ the random matrix $(r_{ij}b_{ij})$.
Note that at every point $\omega$ of the probability space the matrix $\Proj_{I_y(\omega)}\bar A(\omega)$ belongs to $\M_y$.
Then, conditioning on $a_{ij}$'s, we get for every $\tau>0$:
\begin{equation}\label{vector bound sym 125}
\P\bigl\{\|\Proj_{I_y}Ay\|\ge\tau\bigr\}
=\P\bigl\{\|\Proj_{I_y}\bar Ay\|\ge\tau\bigr\}\le\sup\limits_{B\in\M_y}\P\bigl\{\|\bar B y\|\ge\tau\bigr\}.
\end{equation}
Note that for each $B\in\M_y$ and $i\le N$, the $i$-th coordinate of the vector $\bar B y$ satisfies in view of Lemma~\ref{Khintchine mod}:
$$\P\bigl\{|\langle \bar B y,e_i\rangle|\ge\tau\bigr\}\le 2\exp\bigl(-c_{\ref{Khintchine mod}}\tau^2/2\bigr),\;\;\tau>0.$$
A standard application of the Laplace transform then yields
$$\P\bigl\{\|\bar B y\|\ge L_{\ref{vector bound sym}}\sqrt{N}\bigr\}\le \exp(-KN)$$
for some $L_{\ref{vector bound sym}}>0$ depending only on $K$.
This, together with \eqref{vector bound sym 125}, proves the result.
\end{proof}

\begin{lemma}\label{card of Iy}
Let $\xi$ be a symmetrically distributed random variable with unit variance. For every $\varepsilon>0$
and $K>0$ there is $\delta_{\ref{card of Iy}}
=\delta_{\ref{card of Iy}}(\varepsilon,K)>0$
depending on $\varepsilon$, $K$ and the distribution of $\xi$ with the following property:
whenever $N,n\in\N$, $N\ge n$; $A=(a_{ij})$ is an $N\times n$ random matrix with i.i.d.\ entries
distributed as $\xi$ and $y\in S^{n-1}$ is a vector satisfying $\|y\|_\infty\le \delta_{\ref{card of Iy}}$,
we have
$$\P\{|I_y|\le N-\varepsilon N\}\le\exp(-K N),$$
where $I_y$ is defined as in Lemma~\ref{vector bound sym}.
\end{lemma}
\begin{proof}
Fix any $K>0$ and $\varepsilon\in(0,1]$. In view of Lemma~\ref{law of ln}, there is $\delta>0$ such 
that for all $y=(y_1,y_2,\dots)\in \ell_2$ with $\|y\|=1$ and $\|y\|_\infty\le\delta$, and for a sequence of independent
random variables $a_1,a_2\dots$ distributed as $\xi$, we have
$$\P\Bigl\{\sum\limits_{j=1}^\infty a_{j}^2 y_j^2> 2\Bigr\}\le \varepsilon\exp\bigl(-1-K/\varepsilon\bigr).$$
Now, fix $N,n\in\N$ with $N\ge n$ and $y\in S^{n-1}$ with $\|y\|_\infty\le\delta$, and let $A$ be defined as above.
Then, using the last estimate, we obtain
\begin{align*}
\P\bigl\{|I_y|\le N-\varepsilon N\bigr\}
&=\P\Bigl\{\bigl|\bigl\{i\le N:\,\sum_{j} a_{ij}^2 y_j^2> 2\bigr\}\bigr|\ge\varepsilon N\Bigr\}\\
&\le {N\choose \lceil \varepsilon N\rceil}
\Bigl(\frac{\varepsilon}{e}\Bigr)^{\lceil \varepsilon N\rceil}\exp(-KN)\\
&\le \exp(-K N).
\end{align*}
\end{proof}

As an elementary consequence of Lemmas~\ref{vector bound sym} and~\ref{card of Iy} we get
\begin{lemma}\label{single v bound}
Let $\xi$ be a symmetrically distributed random variable with unit variance. For every $\varepsilon>0$
and $K>0$ there are
$\delta_{\ref{single v bound}}=\delta_{\ref{single v bound}}(\varepsilon,K)>0$
depending on $\varepsilon$, $K$ and the distribution of $\xi$, and
$L_{\ref{single v bound}}=L_{\ref{single v bound}}(K)>0$ depending only on $K$ such that, whenever $N,n\in\N$, $N\ge n$;
$A=(a_{ij})$ is an $N\times n$ random matrix with i.i.d.\ entries distributed as $\xi$, and
$y\in S^{n-1}$ is a vector satisfying $\|y\|_\infty\le \delta_{\ref{single v bound}}$, we have
$$\P\bigl\{\min\limits_{|I|\ge N-\varepsilon N}\|\Proj_{I}Ay\|\ge  L_{\ref{single v bound}}\sqrt{N}\bigr\}\le\exp(-KN).$$
\end{lemma}

\begin{lemma}\label{cube bound}
Let $\xi$ be a symmetrically distributed random variable with unit variance. For every $\varepsilon>0$
and $K>0$ there are $n_{\ref{cube bound}}=n_{\ref{cube bound}}(\varepsilon,K)\in\N$
depending on $\varepsilon$, $K$ and the distribution of $\xi$, and
$L_{\ref{cube bound}}=L_{\ref{cube bound}}(K)>0$ depending only on $K$ such that, whenever $N\ge n\ge n_{\ref{cube bound}}$ and
$A=(a_{ij})$ is an $N\times n$ random matrix with i.i.d.\ entries distributed as $\xi$, we have
$$\P\bigl\{\min\limits_{|I|\ge N-\varepsilon N}\max\limits_{y\in B_\infty^n}
\|\Proj_{I}Ay\|\ge  L_{\ref{cube bound}}\sqrt{nN}\bigr\}\le\exp(-KN).$$
\end{lemma}
\begin{proof}
Fix any $K>0$ and $\varepsilon>0$ and define
$n_{\ref{cube bound}}=\lceil\delta_{\ref{card of Iy}}(\varepsilon,K+1)^{-2}\rceil$,
where $\delta_{\ref{card of Iy}}>0$ is taken from Lemma~\ref{card of Iy}.
Now, choose any $N\ge n\ge n_{\ref{cube bound}}$ and let
$A=(a_{ij})$ be an $N\times n$ random matrix with i.i.d.\ entries
distributed as $\xi$.
Let $V$ be the set of vertices of the cube $\frac{1}{\sqrt{n}}B_\infty^n=[-\frac{1}{\sqrt{n}},\frac{1}{\sqrt{n}}]^n$.
In view of Lemma~\ref{card of Iy}, any $v\in V$ satisfies
$$\P\{|I_v|\le N-\varepsilon N\}\le\exp\bigl(-(K+1) N\bigr).$$
Next, by Lemma~\ref{vector bound sym}, for $L=L_{\ref{vector bound sym}}(K+2)>0$ we have
$$\P\bigl\{\|\Proj_{I_v}Av\|\ge  L\sqrt{N}\bigr\}\le\exp\bigl(-(K+2)N\bigr)$$
for all $v\in V$. Note that for any $u,v\in V$ the random sets $I_u$ and $I_v$ coincide everywhere on $\Omega$.
Hence, together with the above estimates, we get
\begin{align*}
\P&\bigl\{\min\limits_{|I|\ge N-\varepsilon N}\max\limits_{v\in V}\|\Proj_{I}Av\|\ge  L\sqrt{N}\bigr\}\\
&\le\exp\bigl(-(K+1) N\bigr)+\P\bigl\{\max\limits_{v\in V}\|\Proj_{I_v}Av\|\ge  L\sqrt{N}\bigr\}\\
&\le\exp(-KN).
\end{align*}
It remains to note that for any $I\subset \{1,2,\dots,N\}$ and $y\in B_\infty^n$ we have
$$\|\Proj_{I}Ay\|\le\sqrt{n}\max\limits_{v\in V}\|\Proj_{I}Av\|$$
everywhere on $\Omega$.
\end{proof}

In the following statement, we bound the quantity \eqref{the quantity} assuming that the matrix entries are
symmetrically distributed. The lemmas above provide estimates for $\min\limits_{|I|\ge N-\varepsilon N}\|\Proj_{I}Ay\|$
for individual vectors on the sphere as well as an upper bound on the cube $\frac{1}{\sqrt{n}}B_\infty^n$.
To derive an estimate for the supremum over the sphere, we shall embed $S^{n-1}$ into Minkowski sum
of a multiple of $B_\infty^n$ and two specially chosen finite sets (see \eqref{weak lsv sym decomp} in the proof below).
This way each vector $y\in S^{n-1}$ can be ``decomposed'' as a sum of three vectors with particular characterestics.
This approach is similar to splitting the unit sphere into sets of ``close to sparse'' and ``far from sparse'' vectors
introduced in \cite{LPRT} and subsequently used in \cite{RV_RECT}, \cite{RV_SQUARE}.

\begin{lemma}\label{weak lsv sym}
Let $\xi$ be a symmetrically distributed random variable with unit variance, and let $\varepsilon\in(0,1]$.
Then there are $N_{\ref{weak lsv sym}}=N_{\ref{weak lsv sym}}(\varepsilon)\in\N$ depending on
$\varepsilon$ and the distribution of $\xi$ and
$w_{\ref{weak lsv sym}}=w_{\ref{weak lsv sym}}(\varepsilon)>0$ depending only on $\varepsilon$ such that,
whenever $N\ge N_{\ref{weak lsv sym}}$,
$n\le N$ and $A=(a_{ij})$ is an $N\times n$ random matrix with i.i.d.\ entries distributed as $\xi$,
we have
$$\P\bigl\{\sup\limits_{y\in S^{n-1}}\min\limits_{|I|\ge N-\varepsilon N}\|\Proj_{I}Ay\|\le C_{\ref{weak lsv sym}}\sqrt{N}\bigr\}
\ge 1-\exp(-w_{\ref{weak lsv sym}} N),$$
where $C_{\ref{weak lsv sym}}>0$ is a universal constant. 
\end{lemma}
\begin{proof}
Fix $\varepsilon\in(0,1]$ and let $N_{\ref{weak lsv sym}}$ be the smallest integer such that
\begin{enumerate}
\item[1)] $\lfloor N_{\ref{weak lsv sym}}^{1/4}\rfloor\delta_{\ref{single v bound}}(\varepsilon/3,2C_{\ref{cubic net}})\ge 1$;
\item[2)] $N_{\ref{weak lsv sym}}\ge\max\bigl(N_{\ref{single vector est}}(\varepsilon/3),n_{\ref{cube bound}}(\varepsilon/3,1)\bigr)$;
\item[3)] for all $N\ge N_{\ref{weak lsv sym}}$,
$$(12e N)^{\sqrt{N}}\exp(-c_{\ref{single vector est}}\varepsilon N/3)+e^{-C_{\ref{cubic net}}N}+e^{-N}
\le \exp\bigl(-\min(c_{\ref{single vector est}}\varepsilon/6,C_{\ref{cubic net}}/2,1/2)N\bigr).$$
\end{enumerate}

Choose $N\ge N_{\ref{weak lsv sym}}$. Without loss of generality, we can assume that $n=N$. Let $A$ be as stated above.

We say that a vector $y\in\R^N$ is $m$-sparse if it has at most $m$ non-zero coordinates.
It is not difficult to verify, using Lemma~\ref{usual net},
that the set of all $\sqrt{N}$-sparse vectors in $2B_2^N$ admits a $N^{-1/2}$-net $\Net_1$
of cardinality at most ${N\choose \lfloor\sqrt{N}\rfloor}(6\sqrt{N})^{\sqrt{N}}\le (12e N)^{\sqrt{N}}$.
Denote
$$T=\bigl\{y\in S^{N-1}:\,\|y\|_\infty\le 1/\lfloor N^{1/4}\rfloor\bigr\}.$$
By Lemma~\ref{cubic net}, there is a finite subset $\Net_2\subset T$ of cardinality at most
$\exp(C_{\ref{cubic net}}N)$ such that for any $y\in T$ there is $y'\in\Net_2$ with
$\|y-y'\|_\infty\le N^{-1/2}$.

Now, we claim that
\begin{equation}\label{weak lsv sym decomp}
S^{N-1}\subset \Net_1+\Net_2+\frac{2}{\sqrt{N}}B_\infty^N,
\end{equation}
i.e.\ any vector $y=(y_1,y_2,\dots,y_N)\in S^{N-1}$ can be represented as $y=y^1+y^2+y^3$
for some $y^1\in \Net_1$, $y^2\in\Net_2$ and $y^3\in \frac{2}{\sqrt{N}}B_\infty^N$.
Indeed, we can always find a subset $J\subset\{1,2,\dots,N\}$
of cardinality $\lfloor \sqrt{N}\rfloor$ such that $|y_j|\le 1/\lfloor N^{1/4}\rfloor$ whenever $j\notin J$.
Denote $r=\sqrt{1-\|y-\Proj_{J}y\|^2}$ and $\tilde y=\Proj_{J}y-r|J|^{-1/2}\sum_{j\in J}e_j$.
Note that $\tilde y$ is $\sqrt{N}$-sparse and has the Euclidean norm at most $2$, so
there is $y^1\in\Net_1$ such that
$\|\tilde y-y^1\|_\infty\le\|\tilde y-y^1\|\le N^{-1/2}$.
Next, the vector $y-\tilde y$ satisfies
$\|y-\tilde y\|=1$ and $\|y-\tilde y\|_\infty\le 1/\lfloor N^{1/4}\rfloor$, i.e.\
$y-\tilde y\in T$. Hence there is $y^2\in\Net_2$
such that $\|y-\tilde y-y^2\|_\infty\le N^{-1/2}$.
Finally, for the vector $y^3=y-y^1-y^2$ we get
$$\|y-y^1-y^2\|_\infty\le \|\tilde y-y^1\|_\infty+\|y-\tilde y-y^2\|_\infty\le \frac{2}{\sqrt{N}},$$
so $y^3\in \frac{2}{\sqrt{N}}B_\infty^N$. This proves \eqref{weak lsv sym decomp}.

For each $y^1\in\Net_1$, in view of Lemma~\ref{single vector est} and the condition $N\ge N_{\ref{single vector est}}(\varepsilon/3)$, we have
$$\P\bigl\{\min\limits_{|I|\ge N-\varepsilon N/3}\|\Proj_{I}Ay^1\|> 2C_{\ref{single vector est}}\sqrt{N}\bigr\}
\le \exp(-c_{\ref{single vector est}}\varepsilon N/3).$$
Next, for every $y^2\in\Net_2$, Lemma~\ref{single v bound} together with the inequality
$\lfloor N^{1/4}\rfloor\delta_{\ref{single v bound}}(\varepsilon/3,2C_{\ref{cubic net}})\ge 1$
and $\|y^2\|_\infty\le 1/\lfloor N^{1/4}\rfloor$ implies that
$$\P\bigl\{\min\limits_{|I|\ge N-\varepsilon N/3}\|\Proj_{I}Ay^2\|\ge  L_{\ref{single v bound}}\sqrt{N}\bigr\}
\le\exp(-2C_{\ref{cubic net}}N)$$
for some constant $L_{\ref{single v bound}}>0$.
Finally, by Lemma~\ref{cube bound} and in view of the condition $N\ge n_{\ref{cube bound}}(\varepsilon/3,1)$ we have
$$\P\bigl\{\min\limits_{|I|\ge N-\varepsilon N/3}\max\limits_{y\in \frac{1}{\sqrt{N}}B_\infty^N}
\|\Proj_{I}Ay\|\ge  L_{\ref{cube bound}}\sqrt{N}\bigr\}\le\exp(-N),$$
where $L_{\ref{cube bound}}>0$ is a universal constant.
Let $\mathcal E$ denote the event
\begin{align*}
\mathcal E
=\Bigl\{\omega\in\Omega:\,&\mbox{for every }y^1\in\Net_1\mbox{ there is a set }I_1=I_1(y^1)\mbox{ with }|I_1|\ge N-\varepsilon N/3\\
&\mbox{such that }\|\Proj_{I_1}A(\omega)y^1\|\le 2C_{\ref{single vector est}}\sqrt{N}\mbox{ {\bf AND}}\\
&\mbox{for every }y^2\in\Net_2\mbox{ there is a set }I_2=I_2(y^2)\mbox{ with }|I_2|\ge N-\varepsilon N/3\\
&\mbox{such that }\|\Proj_{I_2}A(\omega)y^2\|\le L_{\ref{single v bound}}\sqrt{N}\mbox{ {\bf AND}}\\
&\mbox{there is a set }I_3\mbox{ with }|I_3|\ge N-\varepsilon N/3\\
&\mbox{such that }\max\limits_{y\in \frac{2}{\sqrt{N}}B_\infty^N}\|\Proj_{I_3}A(\omega)y\|\le  2L_{\ref{cube bound}}\sqrt{N}\Bigr\}.
\end{align*}
Then from the above probability estimates and the definition of $N_{\ref{weak lsv sym}}$ we obtain
$$\P(\mathcal E)\ge 1-(12e N)^{\sqrt{N}}\exp(-c_{\ref{single vector est}}\varepsilon N/3)
-\exp(-C_{\ref{cubic net}}N)-\exp(-N)\ge 1-\exp(-w_{\ref{weak lsv sym}}N),$$
where $w_{\ref{weak lsv sym}}=\min\bigl(\frac{c_{\ref{single vector est}}\varepsilon}{6},\frac{C_{\ref{cubic net}}}{2},\frac{1}{2}\bigr)$.

Finally, take any $\omega\in\mathcal E$ and any $y\in S^{N-1}$, and let
$y^1\in\Net_1$, $y^2\in\Net_2$ and $y^3\in \frac{2}{\sqrt{N}}B_\infty^N$ satisfy $y=y^1+y^2+y^3$.
Then, by the definition of $\mathcal E$, there are sets $I_1,I_2,I_3\subset\{1,2,\dots,N\}$ with
$|I_\ell|\ge N-\varepsilon N/3$ ($\ell=1,2,3$) such that
\begin{align*}
&\|\Proj_{I_1}A(\omega)y^1\|\le 2C_{\ref{single vector est}}\sqrt{N};\\
&\|\Proj_{I_2}A(\omega)y^2\|\le L_{\ref{single v bound}}\sqrt{N};\\
&\|\Proj_{I_3}A(\omega)y^3\|\le  2L_{\ref{cube bound}}\sqrt{N}.
\end{align*}
Note that the intersection $I=I_1\cap I_2\cap I_3$ necessarily satisfies $|I|\ge N-\varepsilon N$,
and from the last inequalities we get
$\|\Proj_{I}A(\omega)y\|\le (2C_{\ref{single vector est}}+L_{\ref{single v bound}}+2L_{\ref{cube bound}})\sqrt{N}$.
Since our choice of $y\in S^{N-1}$ and $\omega\in\mathcal E$ was arbitrary, we get
$$\P\bigl\{\sup\limits_{y\in S^{N-1}}\min\limits_{|I|\ge N-\varepsilon N}\|\Proj_{I}Ay\|\le
(2C_{\ref{single vector est}}+L_{\ref{single v bound}}+2L_{\ref{cube bound}})\sqrt{N}\bigr\}
\ge \P(\mathcal E)\ge 1-\exp\bigl(-w_{\ref{weak lsv sym}} N\bigr).$$
\end{proof}

Finally, we can state the main result of the section.

\begin{prop}\label{weak lsv nonsym}
Let $\xi$ be a random variable with zero mean and unit variance, and let $\varepsilon\in(0,1]$.
Then there are $N_{\ref{weak lsv nonsym}}=N_{\ref{weak lsv nonsym}}(\varepsilon)\in\N$
depending on $\varepsilon$ and the distribution of $\xi$ and
$w_{\ref{weak lsv nonsym}}=w_{\ref{weak lsv nonsym}}(\varepsilon)>0$
depending only on $\varepsilon$ such that,
whenever $N\ge N_{\ref{weak lsv nonsym}}$,
$n\le N$ and $A=(a_{ij})$ is an $N\times n$ random matrix with i.i.d.\ entries distributed as $\xi$,
we have
$$\P\bigl\{\sup\limits_{y\in S^{n-1}}\min\limits_{|I|\ge N-\varepsilon N}\|\Proj_{I}Ay\|\le C_{\ref{weak lsv nonsym}}\sqrt{N}\bigr\}
\ge 1-\exp(-w_{\ref{weak lsv nonsym}} N),$$
where $C_{\ref{weak lsv nonsym}}>0$ is a universal constant. 
\end{prop}
\begin{proof}
Fix any $\varepsilon\in(0,1]$ and let $\xi'$ be an independent copy of $\xi$.
Then $\frac{1}{\sqrt{2}}(\xi-\xi')$ is symmetrically
distributed and $\Exp\bigl(\frac{1}{\sqrt{2}}(\xi-\xi')\bigr)^2=1$.
Let $N_{\ref{weak lsv sym}},w_{\ref{weak lsv sym}}$ from Lemma~\ref{weak lsv sym} be defined with respect to $\varepsilon$
and the distribution of $\frac{1}{\sqrt{2}}(\xi-\xi')$, and let
$N_{\ref{weak lsv nonsym}}$ be the smallest integer greater than 
$N_{\ref{weak lsv sym}}$  such that $\exp(w_{\ref{weak lsv sym}} N_{\ref{weak lsv nonsym}}/2)\ge \frac{4}{3}$.
Take any $N\ge N_{\ref{weak lsv nonsym}}$ and $n\le N$ and
let $A$ be an $N\times n$ random matrix with i.i.d.\ entries distributed as $\xi$, and $A'$ be an independent copy of $A$.
We can find a Borel function $f:\R^{N\times n}\to S^{n-1}$ such that for any $B\in\R^{N\times n}$ we have
$$\min\limits_{|I|\ge N-\varepsilon N}\|\Proj_I B f(B)\|
\ge\sup\limits_{y\in S^{n-1}}\min\limits_{|I|\ge N-\varepsilon N}\|\Proj_{I}By\|- 1$$
(the term ``$-1$'' above allows us to construct a piecewise constant function $f$, thus
avoiding any measurability questions).
Then we define a random vector $\tilde Y:\Omega\to S^{n-1}$ as $\tilde Y(\omega)=f(A(\omega))$.
Conditioning on $A$, we obtain
\begin{align*}
\P&\bigl\{\min\limits_{|I|\ge N-\varepsilon N}\|\Proj_{I}A\tilde Y\|
>(\sqrt{2}C_{\ref{weak lsv sym}}+2)\sqrt{N}\mbox{ {\bf and} }\|A'\tilde Y\|\le 2\sqrt{N}\bigr\}\\
&\ge \min\limits_{y\in S^{n-1}}\P\bigl\{\|A'y\|\le 2\sqrt{N}\bigr\}
\,\P\bigl\{\min\limits_{|I|\ge N-\varepsilon N}\|\Proj_{I}A\tilde Y\|>
(\sqrt{2}C_{\ref{weak lsv sym}}+2)\sqrt{N}\bigr\}\\
&\ge \frac{3}{4}\P\bigl\{\min\limits_{|I|\ge N-\varepsilon N}\|\Proj_{I}A\tilde Y\|>
(\sqrt{2}C_{\ref{weak lsv sym}}+2)\sqrt{N}\bigr\}.
\end{align*}
Hence, taking into consideration that the entries of $A-A'$ are distributed as $\xi-\xi'$ and
using Lemma~\ref{weak lsv sym}, we get
\begin{align*}
\P&\bigl\{\sup\limits_{y\in S^{n-1}}\min\limits_{|I|\ge N-\varepsilon N}\|\Proj_{I}Ay\|>
(\sqrt{2}C_{\ref{weak lsv sym}}+3)\sqrt{N}\bigr\}\\
&\le\P\bigl\{\min\limits_{|I|\ge N-\varepsilon N}\|\Proj_{I}A\tilde Y\|>
(\sqrt{2}C_{\ref{weak lsv sym}}+2)\sqrt{N}\bigr\}\\
&\le\frac{4}{3}\P\bigl\{\min\limits_{|I|\ge N-\varepsilon N}\|\Proj_{I}A\tilde Y\|>
(\sqrt{2}C_{\ref{weak lsv sym}}+2)\sqrt{N}\mbox{ and }\|A'\tilde Y\|\le 2\sqrt{N}\bigr\}\\
&\le\frac{4}{3}\P\bigl\{\min\limits_{|I|\ge N-\varepsilon N}\|\Proj_{I}(A-A')\tilde Y\|>
\sqrt{2}C_{\ref{weak lsv sym}}\sqrt{N}\bigr\}\\
&\le\frac{4}{3}\P\bigl\{\sup\limits_{y\in S^{n-1}}\min\limits_{|I|\ge N-\varepsilon N}\|\Proj_{I}(A-A')y\|>
\sqrt{2}C_{\ref{weak lsv sym}}\sqrt{N}\bigr\}\\
&\le\frac{4}{3}\exp(-w_{\ref{weak lsv sym}} N)\\
&\le\exp(-w_{\ref{weak lsv sym}} N/2).
\end{align*}
\end{proof}

\section{Matrix truncation and proof of Theorem~\ref{universal theor}}

In the next statement, we compare the $n$-th largest singular value of a random $N\times n$ matrix $A$ with bounded entries to
$s_n(\Proj_I A)$. Obviously,
$$s_n(\Proj_I A)\le s_n(A)\;\;\mbox{for any }I\subset\{1,2,\dots,N\}.$$
We will need an inequality in the opposite direction when $|I|/N\approx 1$.
A theorem of Litvak, Pajor, Rudelson and Tomczak-Jaegermann from \cite{LPRT} implies that for any $\delta>1$
and $M>0$ there are $h>0$ and $\varepsilon>0$ depending only on $\delta$ and $M$ with the following property:
whenever $N\ge\delta n$ and $A$ is an $N\times n$ random matrix with i.i.d.\ entries with mean zero,
variance one and a.s.\ bounded by $M$, we have
$$
\P\bigl\{\min\limits_{|I|\ge N-\varepsilon N}s_n(\Proj_I A)\ge h\sqrt{N}\bigr\}\ge 1-2\exp(\varepsilon N).
$$
This, together with an upper bound for $s_n(A)$, gives an estimate
$$s_n(A)\le L\min\limits_{|I|\ge N-\varepsilon N}s_n(\Proj_I A)$$
with a large probability, where $L>0$ depends only on $\delta$ and $M$. However,
such an estimate would be insufficient
for our needs, and we shall apply a more direct argument to get a stronger relation.

\begin{prop}\label{ssv of submatr}
Let $\xi$ be a random variable with zero mean such that $|\xi|\le M$ a.s.\ for some $M>0$. For any $\eta>0$ there are
$\varepsilon_{\ref{ssv of submatr}}=\varepsilon_{\ref{ssv of submatr}}(\eta,M)>0$
and $N_{\ref{ssv of submatr}}=N_{\ref{ssv of submatr}}(\eta,M)\in\N$
(both depending only on $\eta$ and $M$) with the following property:
whenever $N\ge N_{\ref{ssv of submatr}}$, $n\le N$ and $A=(a_{ij})$ is an $N\times n$ random matrix with i.i.d.\ entries
distributed as $\xi$, we have
$$
\P\bigl\{s_n(A)\le \min\limits_{|I|\ge N-\varepsilon_{\ref{ssv of submatr}} N}s_n(\Proj_I A)+\eta\sqrt{N}\bigr\}
\ge 1-\exp(-\varepsilon_{\ref{ssv of submatr}} N).
$$
\end{prop}
\begin{proof}
Fix any $\eta>0$, let $\varepsilon=\varepsilon_{\ref{ssv of submatr}}(\eta,M)$ be the largest number in $(0,1]$ satisfying
$$\frac{c_{\ref{Khintchine mod}}}{2M^2}\eta^2\ge\varepsilon\bigl(1+\ln\frac{6e}{\varepsilon}\bigr),$$
and $N_{\ref{ssv of submatr}}\in\N$ be the smallest number such that $N-\lceil N-\varepsilon N\rceil\ge\varepsilon N/2$
for all $N\ge N_{\ref{ssv of submatr}}$.

Let $N\ge N_{\ref{ssv of submatr}}$,
$n\le N$ and $A$ be an $N\times n$ random matrix defined as above. We shall prove the statement by contradiction.
Let us assume that
$$
\P\bigl\{s_n(A)>\min\limits_{|I|\ge N-\varepsilon N}s_n(\Proj_I A)+\eta\sqrt{N}\bigr\}>\exp(-\varepsilon N).
$$
Cardinality of the set $T=\bigl\{I\subset\{1,2,\dots,N\}:\,|I|=\lceil N-\varepsilon N\rceil\bigr\}$
can be estimated as
$$|T|\le {N\choose \lceil N-\varepsilon N\rceil}
\le\Bigl(\frac{eN}{N-\lceil N-\varepsilon N\rceil}\Bigr)^{N-\lceil N-\varepsilon N\rceil}
\le \Bigl(\frac{2e}{\varepsilon}\Bigr)^{\varepsilon N}.$$
Hence, our assumption implies that there is a set $I_0\in T$ such that
\begin{equation}\label{ssv of submatr 1}
\P\bigl\{s_n(A)>s_n(\Proj_{I_0} A)+\eta\sqrt{N}\bigr\}\\
>\exp(-\varepsilon N)\Bigl(\frac{2e}{\varepsilon}\Bigr)^{-\varepsilon N}.
\end{equation}
Let $f:\R^{N\times n}\to S^{n-1}$ be a Borel function such that for every $B\in\R^{N\times n}$,
$f(B)\in S^{n-1}$ is an eigenvector of $B^T B$ corresponding to its smallest eigenvalue. So, we have $\|Bf(B)\|=s_n(B)$.
Then we define a random vector $\tilde Y:\Omega\to S^{n-1}$ as $\tilde Y(\omega)=f(\Proj_{I_0}A(\omega))$.
It is not difficult to see that such a definition implies that $\tilde Y$ and $a_{ij}$ ($i\notin I_0$, $1\le j\le n$)
are jointly independent. Hence,
\begin{align*}
\P\bigl\{s_n(A)>s_n(\Proj_{I_0} A)+\eta\sqrt{N}\bigr\}
&\le \P\bigl\{\|A\tilde Y\|>\|\Proj_{I_0} A\tilde Y\|+\eta\sqrt{N}\bigr\}\\
&\le \P\bigl\{\|\Proj_{\{1,2,\dots,N\}\setminus I_0} A\tilde Y\|>\eta\sqrt{N}\bigr\}\\
&\le \sup\limits_{y\in S^{n-1}}\P\bigl\{\|\Proj_{\{1,2,\dots,N\}\setminus I_0} Ay\|>\eta\sqrt{N}\bigr\}.
\end{align*}
Now, for every $y=(y_1,y_2,\dots,y_n)\in S^{n-1}$, Lemma~\ref{Khintchine mod} and
the standard procedure with the Laplace transform
give for $\lambda=\frac{c_{\ref{Khintchine mod}}}{2M^2}$:
\begin{align*}
\P&\bigl\{\|\Proj_{\{1,2,\dots,N\}\setminus I_0} Ay\|>\eta\sqrt{N}\bigr\}\\
&=\P\Bigl\{\sum\limits_{i\notin I_0}\Bigl(\sum\limits_{j=1}^n a_{ij}y_j\Bigr)^2>\eta^2 N\Bigr\}\\
&\le\frac{\Bigl(\Exp\exp\bigl(\lambda\bigl(\sum_{j=1}^n a_{1j}y_j\bigr)^2\bigr)\Bigr)^{N-\lceil N-\varepsilon N\rceil}}{\exp(\lambda\eta^2 N)}\\
&=\exp(-\lambda\eta^2 N)\Bigl(1+\int_1^\infty\P\Bigl\{\Bigl|\sum_{j=1}^n a_{1j}y_j\Bigr|
\ge\sqrt{\ln\tau/\lambda}\Bigr\}\,d\tau\Bigr)^{N-\lceil N-\varepsilon N\rceil}\\
&\le\exp(-\lambda\eta^2 N)\Bigl(1+2\int_1^\infty
\exp\Bigl(-\frac{c_{\ref{Khintchine mod}}\ln\tau}{\lambda M^2}\Bigr)\,d\tau\Bigr)^{N-\lceil N-\varepsilon N\rceil}\\
&=\exp(-\lambda\eta^2 N)\,3^{N-\lceil N-\varepsilon N\rceil}\\
&\le \exp\bigl(-\lambda\eta^2 N+\varepsilon N\ln 3\bigr).
\end{align*}
Together with \eqref{ssv of submatr 1}, the last estimate implies
$$-\lambda\eta^2+\varepsilon\ln 3>-\varepsilon -\varepsilon\ln\frac{2e}{\varepsilon}.$$
However, this contradicts to our choice of $\varepsilon$. Thus, the initial assumption was wrong,
and the statement is proved.
\end{proof}

Let $\xi$ be a random variable with zero mean.
Then for any $M>0$ we call the variable
$$\xi\chi_{\{|\xi|\le M\}}-\Exp(\xi\chi_{\{|\xi|\le M\}})$$
{\it the centered $M$-truncation of $\xi$}. Here,
$\chi_{\{|\xi|\le M\}}$ is the indicator of the event $\bigl\{\omega\in\Omega:\,|\xi(\omega)|\le M\bigr\}$.

Denote $\tilde\xi_M=\xi\chi_{\{|\xi|\le M\}}-\Exp(\xi\chi_{\{|\xi|\le M\}})$ and
$\theta_M=\xi-\tilde\xi_M=\xi\chi_{\{|\xi|> M\}}+\Exp(\xi\chi_{\{|\xi|\le M\}})$.
Obviously, $\Exp\tilde\xi_M=\Exp\theta_M=0$ and $|\tilde\xi_M|\le 2M$ everywhere on $\Omega$ for any $M>0$.
Further, if the second moment of $\xi$ is bounded then
\begin{align*}
&\Exp{\tilde\xi_M}^2=\Exp(\xi\chi_{\{|\xi|\le M\}})^2-\bigl(\Exp(\xi\chi_{\{|\xi|\le M\}})\bigr)^2
\longrightarrow \Exp\xi^2\;\;\mbox{and}\\
&\Exp{\theta_M}^2=\Exp\xi^2-2\Exp\bigl(\xi^2\chi_{\{|\xi|\le M\}}\bigr)
+\Exp{\tilde\xi_M}^2\longrightarrow 0\;\;\mbox{when}\;M\to\infty.
\end{align*}

\begin{theor}\label{trunc ssv theor}
Let $\xi$ be a random variable with zero mean and unit variance. For any $M>0$ and $\eta>0$ there
are $N_{\ref{trunc ssv theor}}\in\N$ depending on $M,\eta$ and the distribution of $\xi$, and $w_{\ref{trunc ssv theor}}>0$
depending only on $M$ and $\eta$ with the following property:
Let $N\ge N_{\ref{trunc ssv theor}}$, $n\le N$ and let $A=(a_{ij})$ be an $N\times n$ random matrix with i.i.d.\ entries distributed
as $\xi$. Further, let $\tilde A$ be an $N\times n$ matrix with the entries
$\tilde a_{ij}=a_{ij}\chi_{\{|a_{ij}|\le M\}}-\Exp(a_{ij}\chi_{\{|a_{ij}|\le M\}})$ and
denote $\theta=\xi\chi_{\{|\xi|> M\}}+\Exp(\xi\chi_{\{|\xi|\le M\}})$.
Then
$$\P\bigl\{s_n(A)\ge s_n(\tilde A)-\eta\sqrt{N}
-C_{\ref{weak lsv nonsym}}\sqrt{N\Exp\theta^2}\bigr\}
\ge 1-\exp(-w_{\ref{trunc ssv theor}}N).$$
\end{theor}
\begin{proof}
Fix any $M>0$ and $\eta>0$ and let $\theta$ be as above.
We will assume that $\P\{\theta=0\}<1$; otherwise the truncation leaves the variable unchanged and
there is nothing to prove. Let $N_{\ref{ssv of submatr}}=N_{\ref{ssv of submatr}}(\eta,2M)$
and $\varepsilon=\varepsilon_{\ref{ssv of submatr}}(\eta,2M)$ be taken from Proposition~\ref{ssv of submatr}.
Let also $N_{\ref{weak lsv nonsym}}$ and $w_{\ref{weak lsv nonsym}}$ be defined as in Proposition~\ref{weak lsv nonsym}
with respect to $\varepsilon$ and the distribution of the ``normalized tail'' $\theta/\sqrt{\Exp\theta^2}$. Now, let
$N_{\ref{trunc ssv theor}}$ be the smallest integer greater than $\max(N_{\ref{ssv of submatr}},N_{\ref{weak lsv nonsym}})$
such that for all $N\ge N_{\ref{trunc ssv theor}}$ we have
$$\exp(-\varepsilon N)+\exp(-w_{\ref{weak lsv nonsym}} N)
\le\exp\bigl(-\min(\varepsilon/2,w_{\ref{weak lsv nonsym}}/2)N\bigr).$$

Take any $N\ge N_{\ref{trunc ssv theor}}$, $n\le N$, and let $A,\tilde A$ be as stated above.
By Proposition~\ref{ssv of submatr}, we have
$$
\P\bigl\{s_n(\tilde A)>
\min\limits_{|I|\ge N-\varepsilon N}s_n(\Proj_I \tilde A)+\eta\sqrt{N}\bigr\}\le \exp(-\varepsilon N),
$$
and, by Proposition~\ref{weak lsv nonsym},
$$\P\bigl\{\sup\limits_{y\in S^{n-1}}\min\limits_{|I|\ge N-\varepsilon N}\|\Proj_{I}(A-\tilde A)y\|>
C_{\ref{weak lsv nonsym}}\sqrt{N\Exp\theta^2}\bigr\}
\le\exp(-w_{\ref{weak lsv nonsym}} N).$$
Combining the two relations, we get
\begin{align*}
\P&\bigl\{s_n(A)< s_n(\tilde A)-\eta\sqrt{N}-C_{\ref{weak lsv nonsym}}\sqrt{N\Exp\theta^2}\bigr\}\\
&\le\P\bigl\{s_n(\tilde A)>\min\limits_{|I|\ge N-\varepsilon N}s_n(\Proj_I \tilde A)+\eta\sqrt{N}\bigr\}\\
&\hspace{0.5cm}
+\P\bigl\{s_n(A)< \min\limits_{|I|\ge N-\varepsilon N}s_n(\Proj_I \tilde A)
-C_{\ref{weak lsv nonsym}}\sqrt{N\Exp\theta^2}\bigr\}\\
&\le \exp(-\varepsilon N)\\
&\hspace{0.5cm}
+\P\bigl\{\exists y\in S^{n-1}:\;\min\limits_{|I|\ge N-\varepsilon N}s_n(\Proj_I \tilde A)-\|Ay\|
> C_{\ref{weak lsv nonsym}}\sqrt{N\Exp\theta^2}\bigr\}\\
&\le \exp(-\varepsilon N)\\
&\hspace{0.5cm}
+\P\bigl\{\exists y\in S^{n-1}:\;\min\limits_{|I|\ge N-\varepsilon N}(\|\Proj_I \tilde Ay\|-\|\Proj_I A y\|)
> C_{\ref{weak lsv nonsym}}\sqrt{N\Exp\theta^2}\bigr\}\\
&\le \exp(-\varepsilon N)+\exp(-w_{\ref{weak lsv nonsym}} N)\\
&\le\exp\bigl(-\min(\varepsilon/2,w_{\ref{weak lsv nonsym}}/2)N\bigr).
\end{align*}
\end{proof}

\begin{proof}[Proof of Theorem~\ref{universal theor}]
Let $\{a_{ij}\}$ $(1\le i,j<\infty)$ be a two-dimensional array of i.i.d.\ random variables with zero mean and unit variance and
let $(N_m)_{m=1}^\infty$ be an integer sequence satisfying $m/N_m\longrightarrow z$ for some $z\in(0,1)$.
Recall that for every $m\in\N$, $A_{m}$ denotes the random $N_m\times m$ matrix with entries $a_{ij}$
$(1\le i\le N_m,1\le j\le m)$.
The Mar\v cenko--Pastur law (see Theorem~\ref{Mar Pas} and Remark~\ref{Mar Pas rem}) implies that
$$\limsup\limits_{m\to\infty}\frac{s_{m}(A_{m})}{\sqrt{N_m}}\le 1-\sqrt{z}\;\;\mbox{almost surely}.$$
Thus, it suffices to prove the lower estimate
$$\liminf\limits_{m\to\infty}\frac{s_{m}(A_{m})}{\sqrt{N_m}}\ge 1-\sqrt{z}\;\;\mbox{a.s.}$$

Now, choose arbitrary $\eta>0$ and let $M>0$ be such that
\begin{align*}
&\Exp\bigl(a_{11}\chi_{\{|a_{11}|\le M\}}-\Exp(a_{11}\chi_{\{|a_{11}|\le M\}})\bigr)^2\ge (1-\eta)^2\;\;\mbox{and}\\
&\Exp\bigl(a_{11}\chi_{\{|a_{11}|> M\}}+\Exp(a_{11}\chi_{\{|a_{11}|\le M\}})\bigr)^2\le \eta^2.
\end{align*}
For every $m\in\N$, let $\tilde A_m$ be the $N_m\times m$ matrix of truncated and centered variables
$\tilde a_{ij}=a_{ij}\chi_{\{|a_{ij}|\le M\}}-\Exp(a_{ij}\chi_{\{|a_{ij}|\le M\}})$ ($1\le i\le N_m,1\le j\le m$).
Theorem~\ref{trunc ssv theor} and the conditions on the sequence $(N_m)_{m=1}^\infty$ imply that
there are $m_0\in\N$ and $w>0$ such that for all $k\ge m_0$
$$\P\bigl\{s_m(A_{m})\ge s_m(\tilde A_m)-(1+C_{\ref{weak lsv nonsym}})\eta\sqrt{N_m}
\mbox{ for all }m\ge k\bigr\}\ge 1-\sum\limits_{m=k}^\infty\exp(-wN_m),$$
where the quantity on the right-hand side goes to $1$ as $k$ tends to infinity.
Hence, we obtain
$$\P\Bigl\{\liminf\limits_{m\to\infty} \frac{s_m(A_{m})}{\sqrt{N_m}}
\ge\liminf\limits_{m\to\infty}\frac{s_m(\tilde A_m)}{\sqrt{N_m}}-(1+C_{\ref{weak lsv nonsym}})\eta\Bigr\}=1.$$
On the other hand, the theorem of Bai and Yin \cite{BY} implies that 
$$\lim\limits_{m\to\infty}\frac{s_{m}\bigl(\tilde A_{m})}{\sqrt{N_m}}\ge (1-\eta)(1-\sqrt{z})\;\;\mbox{a.s.}$$
Thus, we come to the estimate
$$\liminf\limits_{m\to\infty}\frac{s_m(A_{m})}{\sqrt{N_m}}\ge (1-\eta)(1-\sqrt{z})-(1+C_{\ref{weak lsv nonsym}})\eta\;\;\mbox{a.s.}$$
Since $\eta>0$ was arbitrary, this proves the result.
\end{proof}

\bigskip

{\bf Acknowledgement.}
I would like to thank my supervisor Dr. Nicole~Tomczak-Jaegermann for support and
for valuable suggestions on the text.

\end{document}